\documentclass[11pt]{amsart}
 \usepackage{geometry}                
\usepackage{amssymb}
\usepackage{amsthm}
\usepackage{amsmath}
\usepackage{amsopn}
\usepackage{array}
\usepackage{verbatim}
\usepackage{url}

\theoremstyle{plain}
\newtheorem{teo}{Teorem}
\newtheorem{theorem}[teo]{Theorem}
\newtheorem{proposition}[teo]{Proposition}
\newtheorem{lemma}[teo]{Lemma}

\theoremstyle{plain}

\theoremstyle{definition}

\theoremstyle{remark}
\newtheorem{remark}[teo]{Remark}
\newtheorem{example}[teo]{Example}

\DeclareMathOperator{\perm}{per}

\begin{document}

\begin{abstract}
We study the asymptotic behavior of permanents of $n \times n$ random matrices $A$ with positive entries. 
We assume that $A$ has either i.i.d.~entries or is a symmetric matrix with the i.i.d.~upper triangle.
Under the assumption that  elements have power law decaying tails, we prove a strong law of large numbers for $\log \perm A$. 
We calculate the values of the limit $\lim_{n \to \infty}\frac{\log \perm A}{n \log n}$ in terms of the exponent of the power law distribution decay, and observe a first order phase transition in the limit as the mean becomes infinite. 
The methods extend to a wide class of rectangular matrices. 
It is also shown that, in finite mean regime, the limiting behavior holds uniformly over all submatrices of linear size.
\end{abstract}

\title{Permanents of heavy-tailed random matrices with positive elements}
\author{Ton\'ci Antunovi\'c}
\address{Mathematics Department, University of California, Los Angeles, CA 90095, USA}
\email{tantunovic@math.ucla.edu}


\maketitle

\section{Introduction}

The permanent of an $m \times n$ matrix $A=(a_{i,j})$ (height $m$ and width $n$) satisfying $m \leq n$ is defined as
\[
\perm A = \sum_{\pi} \prod_{i=1}^n a_{i,\pi(i)},
\]
where the sum $\sum_\pi$ goes over the set of one-to-one functions from $[m]=\{1,\dots,m\}$ to $[n]=\{1,\dots,n\}$. 
When $m=n$, that is when $A$ is a square matrix, $\pi$ goes over $S_n$, the set of permutations on $\{1,\dots, n\}$.

In this paper we will study asymptotics of permanents of large random matrices with positive elements.
Permanents of random matrices have been studied in a number of papers. 
Already in \cite{Girko71} and \cite{Girko72} Girko studied the asymptotic behavior of permanent under certain conditions on the characteristic function of the entries of $A$. 
Working in the context of perfect matchings on random bipartite graphs, Janson \cite{Janson94}  proved central limit theorems for permanents of matrices with $0$-$1$ i.i.d.~elements.
In a series of papers Rempa{\l}a and Weso{\l}owski studied the permanents of large rectangular matrices with i.i.d.~columns containing non-zero mean and finite variance elements (allowing some correlation among elements in each column). In the case of i.i.d.~elements, relying on earlier results of van Es and Helmers \cite{EsHelmers88} and Borovskikh and Korolyuk \cite{KB92}, 
they proved central limit theorems   for $(\perm A) / \mathbb{E}(\perm A)$ \cite{RempalaWesolowski99},
and later certain strong laws of large numbers \cite{RW02}, when the height of the matrix grows much slower than the width. 
See also Chapter 3 in \cite{RempalaWesolowski08} for a self-contained discussion of these results.
When elements have zero mean much less is known.
The only  result of this type is  \cite{TaoVu08} by Tao and Vu, who obtained significantly different behavior for $n \times n$ matrices $A$ with independent (mean zero) Bernoulli $\pm 1$ elements. 
They showed that with high probability $|\perm(A)| = n^{(\frac{1}{2}+o(1))n}$.
The significance of this result lies also in the fact that this was the first proof of the fact that for i.i.d.~Bernoulli $\pm 1$ random matrices $\mathbb{P}(\perm(A) = 0) \to 0$.
It doesn't seem that anything substantial is known about the random symmetric matrices. 
In \cite{TaoVu08} it is mentioned that their method does not extend to tackle the symmetric case.
In particular, in \cite{TranVu} it is explicitly conjectured that $|\perm(A)| = n^{(\frac{1}{2}+o(1))n}$, when $A$ is a symmetric mean zero Bernoulli $\pm 1$ matrix.


The above results demonstrate the contrast between the non-zero mean, finite variance case and the Bernoulli case, which can be summarized as
\begin{equation}\label{eq:old_limits}
\lim_{m,n \to \infty}\frac{\log |\perm A|}{m \log n} = \left\{\begin{array}{r l} 1, & \text{ in the case of the finite variance and non zero mean,} \\  \frac{1}{2}, & \text{ in the Bernoulli case with zero mean for } m=n. \end{array}\right.
\end{equation}
In the non-zero finite mean case ($\mu$ being mean of the elements) the value of the limit, and especially the upper bounds, can be inferred by calculating the first moment $\mathbb{E}(\perm A)=\binom{n}{m}m!\mu^n$, and in the Bernoulli case from the second moment $\mathbb{E}(|\perm A|^2)=n!$. 
In this paper we will calculate the value of this limit under the assumption that elements are positive and have power law decaying tails $\mathbb{P}(\xi \geq t) = t^{-1/\beta +o(1)}$. In the case $\beta > 1$ elements have infinite mean which prevents us from making any guesses about the the value of the limit. 
In particular, neither the lower or the upper bound are trivial.
In Theorem \ref{thm:main} we will observe a first order phase transition in the limit at  $\beta =1$, when the mean becomes infinite. 
Our results hold when $A$ is a random symmetric matrix as well $A$.

\section{Setup and the Result}
In the text we will assume that for $m \leq n$, $A_{m,n} = (\xi_{i,j})$ is an $m \times n$ matrix ($A_n$ when $m=n$), whose elements $\xi_{i,j}$ are identically distributed random variables (distributed as $\xi$).
We assume that either the entries $(\xi_{i,j})_{i,j}$ of $A_{m,n}$ are independent, or that $A_n$ is a symmetric matrix such that the entries in the upper triangle $(\xi_{i,j})_{i \leq j}$ are independent.
Note that we will drop the subscripts from $A$ when there is no confusion.
Assuming that the matrices are constructed on a common probability space, theorems below give strong laws of large numbers for $\frac{\log \perm A_n}{n\log n}$. 
Extensions to the i.i.d.~case for rectangular matrices are given in Section \ref{sec:non-square}.
Note that the almost sure convergence in Theorems \ref{thm:main} and \ref{thm:general_bounds} simply means that the estimates we obtain in the proofs are summable.

To summarize, in the following we will assume that we are given a sequence of $n \times n$  random matrices $A_n = (\xi_{i,j})_{i,j}$ defined on a common probability space, whose elements $\xi_{i,j}$ are strictly positive and identically distributed as $\xi$. Moreover we assume that either
\begin{itemize}
\item[I)] for each $n$, the elements $(\xi_{i,j})_{1\leq i,j\leq n}$ of $A_n$ are independent, or
\item[II)] for each $n$, the elements of $A_n$ in the upper triangle $(\xi_{i,j})_{1\leq i \leq j\leq n}$ are independent and $A_n$ is symmetric, that is $\xi_{i,j} = \xi_{j,i}$.
\end{itemize}

\begin{theorem}\label{thm:main}
Assume that a sequence of random matrices $A_n$ satisfies either I) or II), and that the distribution of their entries satisfies
\begin{equation}\label{eq:heavy_tail_assumption}
\lim_{t \to \infty}\frac{\log \mathbb{P}(\xi \geq t)}{\log t} = - \frac{1}{\beta},
\end{equation}
for some $\beta > 0$. 
Then we have
\begin{equation}\label{eq:main_result}
\lim_{n \to \infty} \frac{\log \perm A_n}{n \log n} = \max(1,\beta).
\end{equation}
\end{theorem}

Random variable $\xi$ in \eqref{eq:heavy_tail_assumption}  has finite variance for $\beta < 1/2$, finite  mean for $\beta < 1$, and infinite mean for $\beta > 1$ when we observe a limit different from the values in \eqref{eq:old_limits}.

The following result generalizes the case $\beta < 1$. It does not require the finite variance assumption and gives the general lower bounds and the upper bounds in the case of finite mean uniformly over all submatrices of linear size. Note that for an $m \times n$ matrix $A=(\xi_{i,j})$ any matrix $B=(\xi_{i,j})_{i\in I,j\in J}$, where $I \subset \{1,\dots,m\}$, $J \subset \{1,\dots,n\}$ is called a submatrix of of $A_n$.

\begin{theorem}\label{thm:general_bounds}
Assume that a sequence of random matrices $A_n$ satisfies either I) or II), and fix $0<\alpha < 1$. 
\begin{itemize}
\item[i)] We have almost surely
\begin{equation}\label{eq:general_lower_bounds}
\liminf_{n \to \infty} \min_{(k,B)}\frac{\log \perm B}{k \log k} \geq 1,
\end{equation}
where the minimum is taken over all integers $\alpha n \leq k \leq n$ and all $k \times k$ submatrices $B$ of $A_n$.
\item[ii)]
If the elements of $A_n$  have a finite mean then almost surely
\begin{equation}\label{eq:general_upper_bounds}
\limsup_{n \to \infty} \max_{(k,B)}\frac{\log \perm B}{k \log k} = 1,
\end{equation}
where the maximum is again taken over all integers $\alpha n \leq k \leq n$ and all $k \times k$ submatrices $B$ of $A_n$.
\end{itemize}
\end{theorem}

Note that by part ii) of the above theorem, Theorem \ref{thm:main} holds for $\beta = 0$ as well.

\begin{remark}\label{rem:mass_at_zero}
It is not difficult to adapt the proof of the lower bounds in Theorem \ref{thm:main}, to show that the convergence from Theorem \ref{thm:main} holds in probability under weaker condition $\xi \geq 0$, that is if we allow the distribution of $\xi$ to have a point mass at zero.
For details see Remark \ref{rem:mass_at_zero_proof}.
However, it is clear that Theorem \ref{thm:general_bounds} i) fails in this case.
To simplify the presentation, we choose to assume $\xi >0$ almost surely.
\end{remark}

Condition \eqref{eq:heavy_tail_assumption} in Theorem \ref{thm:main} is satisfied with $\beta > 1$ for many common heavy tail distributions including Pareto distribution, L\'{e}vy distribution, Inverse-Gamma distribution, Beta-prime distribution.
We are particularly interested in the case when $\xi$ has Pareto distribution with parameter $\beta$, that is $\mathbb{P}(\xi \geq t) =t^{-1/\beta}$ for $t \geq 1$. Actually in Section \ref{sec:upper_bounds} the upper bounds in Theorem \ref{thm:main} will be proven in the Pareto case, and then extended to the general case via simple stochastic domination. Note that when the convergence \eqref{eq:heavy_tail_assumption} fails to hold, one cannot guarantee the existence of the limit in \eqref{eq:main_result} (see Example \ref{ex:no_convergence} in Section \ref{sec:non-square}). However, the upper bound on the $\limsup$ in \eqref{eq:heavy_tail_assumption} will imply the upper bound in \eqref{eq:main_result} and similarly the lower bound for $\liminf$. 



\begin{remark}\label{rem:matchings}
From a more combinatorial point of view, permanents can be interpreted in the context of saturated matchings (or perfect matchings for $m=n$) of bipartite graphs. For a bipartite graph $G=(V,E)$, let $V=V_1 \cup V_2$, $|V_1| \leq |V_2|$ be a decomposition of the vertex sets into subsets, so that no two vertices in $V_i$ are connected by an edge, $i=1,2$.
Saturated matchings of $G$ can be defined as subsets $\mathcal{M} \subset E$ of the edge set with the property that every vertex in $V_1$ is adjacent to exactly one edge in $\mathcal{M}$, and that every vertex in  $V_2$ is adjacent to at most one edge in $\mathcal{M}$.
For $m \leq n$ and an $m\times n$ matrix $A=(a_{i,j})$ containing only elements $0$ and $1$, construct a bipartite graph $G$ with $m+n$ vertices $\{v_1, \dots v_m, w_1, \dots w_n\}$ so that $v_i$ and $w_j$ are connected by an edge if and only if $a_{ij}=1$. Clearly every one-to-one function $\pi\colon \{1,\dots,m\} \to \{1,\dots,n\}$ for which $\prod_{i=1}^n a_{i,\pi(i)} = 1$  corresponds to a saturated matching on $G$ in $1$-$1$ manner. Therefore $\perm A$ is equal to the number of saturated matchings on $G$. For general matrices $A$ one can construct the graph by drawing an edge between $v_i$ and $w_j$ whenever $a_{i,j} \neq 0$ and  putting the weight $a_{i,j}$ on this edge. Then $\perm A$ can be interpreted as the total weight of all the saturated matchings on $G$ (a weight of a matching being the product of the weight on its edges). 
This also known as the partition function of the dimer model.
All the results in this paper can be interpreted in this way. 
\end{remark}

In the following section we prove the upper bounds in Theorem \ref{thm:main}, and in Section  \ref{sec:lower_bounds} we provide the lower bounds and prove Theorem \ref{thm:general_bounds}. In the last section we will extend the results to a class of rectangular matrices, and show an example demonstrating that, in general without \eqref{eq:heavy_tail_assumption}, Theorem \ref{thm:main} fails to hold.

\section{Proof of the upper bounds in Theorem \ref{thm:main}}\label{sec:upper_bounds}

The exact calculations needed for the proof of the upper bounds in Theorem \ref{thm:main} are easier to perform when we are given a concrete distribution of $\xi$ to work with. The proof will be provided for the Pareto case, but first we will see how this yields the  upper bounds in Theorem \ref{thm:main} for the general case.

\begin{remark}\label{rem:stochastic_domination}
Throughout the paper we will use the following two simple observations.

\emph{i)}
For any $m \times n$ matrix $A$ and $\lambda \in \mathbb{R}$ we have that $\perm (\lambda A) = \lambda^m \perm A$. Thus the value of the limit of $\log \perm A_{m,n} /(m \log n )$ in Theorem \ref{thm:main} (as well as $\liminf$ and $\limsup$) is unchanged if we replace the random variable $\xi$ by a random variable $\lambda \xi$, for any $\lambda >0$.  

\emph{ii)}
Given two random variables $\xi_1$ and $\xi_2$ such that $\xi_2$ stochastically dominates $\xi_1$ from above, it is possible to extend a probability space on which $\xi_1$ is defined, and construct $\overline{\xi}_2$ with the same distribution as $\xi_2$ on the extended space such that $\xi_1 \leq \overline{\xi}_2$ pointwise.
For example, if $F_i$ is the cumulative distribution function of $\xi_i$, and if the probability space $(\Omega,\mathbf{P})$ supports $\xi_1$, then consider the probability space $(\Omega \times [0,1], \mathbf{P}\times d\lambda)$, where $d\lambda$ is the Lebesgue measure on $[0,1]$.
On this space define $u = \lambda F_1(\xi_1) + (1-\lambda)F_1^-(\xi_1)$, where $F_1^-(t) = \lim_{s \uparrow t}F_1(s)$ and $\overline{\xi}_2 = \inf\{t:F_2(t)  \geq u\}$.
Then it is easily checked that $u$ has uniform distribution, $\overline{\xi}_2$ has the same distribution as $\xi_2$ and $F_1^-(\xi_1) \leq u \leq F_2(\overline{\xi_2})$, which implies that $\xi_1 \leq \overline{\xi}_2$ pointwise.
This extends to random matrices.
In all the models we consider, one can extend a probability space on which $(A_{m,n}^1)$ is a sequence of $m \times n$ random matrices with elements distributed as $\xi_1$, to construct a sequence $(A_{m,n}^2)$ of $m \times n$ random matrices with elements distributed as $\xi_2$, such that elements of $(A_{m,n}^2)$ dominate the corresponding elements of $(A_{m,n}^1)$ pointwise. 
Then $\perm A_{m,n}^1 \leq \perm A_{m,n}^2$, since all their elements are positive.

\end{remark}


\begin{proof}[Proof of the upper bounds in Theorem \ref{thm:main} assuming it holds for the Pareto case]
Fix $\epsilon>0$ and take $M>1$  so that $\mathbb{P}(\xi \geq t) \leq t^{-1/(\beta+\epsilon)}$ holds for all $t \geq M$. Denote by $\overline{\xi}_{\beta+\epsilon}$ a Pareto distributed random variable with parameter $\beta+\epsilon$ and observe that $\mathbb{P}(\xi \geq t) \leq \mathbb{P}(M\overline{\xi}_{\beta+\epsilon} \geq t)$ holds for all $t$. Assuming the statement holds for the Pareto case, Remark \ref{rem:stochastic_domination} ii) implies that almost surely
\[
\limsup_n \frac{\log \perm A_n}{n \log n} \leq \beta + \epsilon.
\]
Since $\epsilon > 0$ was arbitrary the claim follows.
\end{proof}

The rest of this section is devoted to the proof of the upper bounds in the Pareto case in which we show explicit calculations. 
A useful observation which we will use extensively is the fact that if $\xi$ is a Pareto distributed random variable with parameter $\beta$, then $Y = (\log \xi)/\beta$ has exponential distribution with rate $1$, that is $\mathbb{P}(Y \geq t)  = e^{-t}$ for $t\geq 0$. We start by proving some basic estimates for maxima of independent exponential random variables.

\begin{lemma}
\label{lemma: Max}
For $n \geq 2$ let $Y_i$, $1 \leq i \leq n$ be independent exponential random variables with rate 1. 
For a fixed $c < 1$ set $ R=c(\max_{1 \leq i \leq n} Y_i)$.
\begin{itemize}
\item[i)] For any  $t>0$ we have
\begin{equation}
\label{eq: Max1}
\mathbb{P}(R \leq t) = \left(1-e^{-t/c}\right)^n \textrm{ and } \  \ \mathbb{P}(R \geq t) \leq ne^{-t/c}.
\end{equation}
\item[ii)] The expectation of $e^{R}$ can be bounded as 
\begin{equation}
\label{eq: Max2}
\mathbb{E}\left(e^R\right) \leq \frac{n^c}{1-c}.
\end{equation}
\end{itemize}
\end{lemma}

\begin{proof}
i) Both claims are straightforward
\[
\mathbb{P} (R \leq t) = \prod_{i=1}^n \mathbb{P} (Y_i \leq  t/c) = \left(1-e^{-t/c}\right)^n, \  \mathbb{P} (R \geq t) \leq \sum_{i=1}^n \mathbb{P}(Y_i \geq t/c) = ne^{-t/c}.
\]

ii) Using the second inequality in (\ref{eq: Max1}) we obtain for $t \geq 1$
$$
\mathbb{P} \left(e^{R} \geq t\right) = \mathbb{P}(R \geq \log t) \leq nt^{-1/c},
$$
from where we get
\[
\mathbb{E}\left(e^R\right) =  \int_0^{\infty}\mathbb{P}\left(e^R \geq t\right) dt \leq n^c +\int_{n^c}^{\infty}nt^{-1/c} dt  = n^c+ n\Big(\frac{1}{c}-1\Big)^{-1}n^{c-1}
 =   \frac{n^c}{1-c} .
\]

\end{proof}

\begin{lemma}\label{lemma:estimates_on_probabilities}
For $n \geq 2$ let $Y_i$, $1 \leq i \leq n$ be independent exponential random variables with rate 1. 
For $0 \leq a < b$, let $p_n(a,b) = \mathbb{P}(a < \sum_{i=1}^n Y_i  \leq b)$
Then 
\[
e^{-a}b^n \geq n!p_n(a,b)  \geq \left\{
\begin{array}{ll}
e^{-b}b^n, & \text{ if } a=0, \\
ne^{-b}(b-a)a^{n-1}, & \text{ if } a>0.
\end{array}
\right.
\]
\end{lemma}

\begin{proof}
Since $\sum_{i=1}^nY_{i}$ is the sum of $n$ independent exponential random variables with mean $1$, it has Gamma density $\frac{x^{n-1}e^{-x}}{(n-1)!}$, for $x>0$.
Therefore  
\[
n! p_n(a,b) =n! \int_{a}^{b} \frac{x^{n-1}e^{-x}}{(n-1)!}\ dx = n \int_{a}^{b} x^{n-1}e^{-x}\ dx.
\]
The upper bound  now follows
\[
n!p_n(a,b) \leq e^{-a}\int_{a}^{b}nx^{n-1}~dx \leq e^{-a}b^n.
\]
For the lower bound
\[
n!p_n(a,b) =n \int_{a}^{b} x^{n-1}e^{-x}~dx \geq e^{-b}\int_{a}^{b}nx^{n-1}dx = e^{-b}(b^n-a^n).
\]
For $a=0$ we have $n!p_n(a,b)  \geq  e^{-b}b^n$, while for $a>0$  using  $b^n \geq a^n + n(b-a)a^{n-1}$ we obtain the corresponding lower bound.
\end{proof}

The idea of the proof of the upper bounds in the Pareto case is to estimate (by evaluating the expectation) the number of permutations $\pi$ for which the product $\prod_{i=1}^n \xi_{i,\pi(i)}$ will lie in some given interval. The key estimate is provided in Lemma \ref{prop: BoundZMax}. We will only consider the intervals not exceeding $(n\sqrt{\log n})^{\beta n}$, since
as the following lemma  shows, the largest product $\prod_{i}\xi_{i,\pi(i)}$ typically does not exceed this value.

In what follows, for an array of random variables $(Y_{i,j})_{1 \leq i,j \leq n}$ it is assumed in the text either that
\begin{itemize}
\item[(A1)] $ (Y_{i,j})_{i,j}$  is an i.i.d.~family of exponential random variables with rate 1, or
\item[(A2)] $ (Y_{i,j})_{i,j}$ is a symmetric array exponential random variables with rate 1, in the sense that $ (Y_{i,j})_{1 \leq i \leq j \leq n}$ is an i.i.d.~family, and $Y_{i,j} = Y_{j,i}$.
\end{itemize}

\begin{lemma}
\label{lemma: MaxPerm}
Assume that $(Y_{i,j})_{i,j}$ satisfies either the condition (A1) or (A2). Then for any $\lambda>0$ 
\begin{equation}
\label{eq: MaxPermStrong}
\sum_{n = 2}^{\infty}\mathbb{P}\left(\max_{\pi \in S_n}\sum_{i=1}^n Y_{i,\pi(i)} \geq n\log n +  \frac{ n \log\log n}{\lambda}\right) < \infty.
\end{equation}
\end{lemma}

\begin{proof}
The (A2) case will be a special case ($\epsilon=1$) of   Lemma \ref{lemma: MaxPerm-symmetric} which follows.
Assume (A1) condition holds.
 For a fixed $c < 1$ denote $R_i := c(\max_{1 \leq j \leq n}Y_{i,j})$. From the definition of $R_i$ it is obvious that $$ \max_{\pi \in S_n} \sum_{i=1}^n Y_{i,\pi(i)} \leq  \frac{1}{c}\left(\sum_{i=1}^n R_i\right).$$ Using the inequality (\ref{eq: Max2}) we have
\begin{multline*}
\mathbb{P} \left(\max_{\pi \in S_n}\sum_{i=1}^n Y_{i,\pi(i)} \geq n\log n + \frac{n \log\log n}{\lambda}\right)  \leq \mathbb{P}\left(\sum_{i=1}^n R_i \geq cn\log n + c\frac{ n \log\log n}{\lambda}\right) \\
\leq \mathbb{E}\left(e^{\sum_{i=1}^nR_{i} - cn\log n - \frac{c}{\lambda} n \log\log n}\right)
 = \left(\frac{\mathbb{E}\left(e^{R_{1}}\right)}{n^c(\log n)^{c/\lambda}}\right)^n = \Big((1-c)(\log n)^{c/\lambda}\Big)^{-n}.
\end{multline*}
For any $c< 1$, the right hand side above is summable in $n$, which proves the lemma.
\end{proof}

The symmetric case (A2) of Lemma \ref{lemma: MaxPerm}  follows from a stronger result in the following lemma (case $\epsilon =1$).
This lemma also holds in the i.i.d.~case (A1), but we don't need this result in that setting.
\begin{lemma}
\label{lemma: MaxPerm-symmetric}
Assume that $(Y_{i,j})_{i,j}$ satisfies the condition (A2) and fix  $\lambda>0$.
For a fixed $0 < \epsilon \leq 1$ let $\mathfrak{B}_n^\epsilon$ denote the event that for some $\epsilon n \leq k \leq  n$ there are indices $1 \leq p_1 < p_2 < \dots < p_k \leq n$ and a permutation $\pi \in S_n$ such that 
\[
\sum_{j=1}^k Y_{p_j,\pi(p_j)} \geq k \log n + \frac{ k \log\log n}{\lambda}.
\]
Then $\sum_{n \geq 2} \mathbb{P}(\mathfrak{B}_n^\epsilon) < \infty$.
\end{lemma}

\begin{proof}
Take $0 < c< 1/2$ and denote the multiple of the  maximal element in the $i$-th row above the diagonal by $R_i' := c(\max_{i \leq j \leq n}Y_{i,j})$. 
Observe that by \eqref{eq: Max2} we have
\begin{equation}\label{eq: MGF_for_rows_symmetric}
\mathbb{E}\big(e^{R_i'}\big) \leq \frac{(n-i+1)^c}{1-c} \leq \frac{n^c}{1-c}.
\end{equation}
For a fixed permutation $\pi$ consider the set
\[
T_\pi = \left\{(i,\pi(i))\ | \ \text{ when } i \leq \pi(i)\right\} \cup \left\{(\pi(i),i)\ | \ \text{ when } i > \pi(i) \right\},
\]
of coordinates of all elements $Y_{i,\pi(i)}$, reflected over the diagonal to the upper triangular part.
Since the matrix $(Y_{i,j})$ is symmetric, we have $\sum_{i=1}^n  Y_{i,\pi(i)} = \sum_{(i,j) \in T_\pi}Y_{i,j}$.
Observe that for every $1 \leq i \leq n$ there are either zero, one or two elements $j$ such that $(i,j) \in T_\pi$.
Reflecting the elements $(p_i,\pi(p_i))$ across the diagonal, we see that the event $\mathfrak{B}_n^\epsilon$ implies that for some $\epsilon n \leq k \leq  n$ and $0 \leq \ell \leq k/2$ there are disjoint sets of indices $\{i_1^{(2)}, i_2^{(2)}, \dots, i_\ell^{(2)}\}$ and $\{i_1^{(1)}, i_2^{(1)}, \dots, i_{k-2\ell}^{(1)}\}$ such that
\[
2\sum_{j = 1}^{\ell}R_{i_j^{(2)}}'  + \sum_{j = 1}^{k-2\ell}R_{i_j^{(1)}}' \geq ck\log n + c\frac{k \log \log n}{\lambda}.
\]
For a fixed $k \geq \epsilon n$, fixed $0 \leq \ell \leq k/2$ and a fixed choice of indices $\{i_1^{(2)}, i_2^{(2)}, \dots, i_\ell^{(2)}\}$ and $\{i_1^{(1)}, i_2^{(1)}, \dots, i_{k-2\ell}^{(1)}\}$, the probability of the event above is bounded from above by 
\[
\frac{\prod_{j = 1}^{\ell}\mathbb{E}\exp\big(2R_{i_j^{(2)}}'\big) \cdot \prod_{j = 1}^{k-2\ell}\mathbb{E}\exp\big(R_{i_j^{(1)}}'\big) }{n^{ck} (\log n)^{c k/\lambda}} \leq \frac{\frac{n^{2c\ell}}{(1-2c)^\ell} \cdot \frac{n^{c(k-2\ell)}}{(1-c)^{k-2\ell}} }{n^{ck} (\log n)^{ck/\lambda }} 
\leq \Big((1-2c)(\log n)^{c/\lambda}\Big)^{-\epsilon n}.
\]
Observe that the right hand side above depends only on $n$ and $c$.
Since one can find $(2^n)^2 = 4^n$ pairs of subsets of $\{1,\dots, n\}$, there are at most $4^n$ choices for $k$, $\ell$ and the sets $\left\{i_1^{(2)}, i_2^{(2)}, \dots, i_\ell^{(2)}\right\}$  and $\left\{i_1^{(1)}, i_2^{(1)}, \dots, i_{k-2\ell}^{(1)}\right\}$.
Thus taking the union bound we see that
\[
\mathbb{P}(\mathfrak{B}_n^\epsilon ) \leq \left(\frac{4^{1/\epsilon}}{(1-2c)(\log n)^{c/\lambda}}\right)^{\epsilon n},
\]
which is clearly summable in $n$.

\end{proof}

\begin{remark}\label{remark:permutation_with_no_transpositions}
For $0 \leq \ell < n/2$, let $S_{n}^\ell$ denote the set of permutations in $S_n$ whose cycle decomposition has exactly $\ell$ transpositions (that is, there are exactly $\ell$ unordered pairs $(i,j)$ with $i \neq j$ such that $\pi(i)=j$ and $\pi(j)=i$).
The number of transpositions in a uniformly random permutation converges to a Poisson distribution of rate $1/2$, see \cite{ArratiaTavare}.
We will need a simple consequence, that the probability that a uniformly chosen random permutation on $n \geq 3$ elements is uniformly bounded from zero, that is there is $a>0$ such that for all $n \geq 3$ we have $|S_{n}^0| \geq an!$.
This in particular shows that for $n\geq 3$
\begin{equation}\label{eq:estimates_on_ptransposition_numbers}
\frac{an!}{2^\ell \ell!} \leq |S_{n}^\ell| \leq \frac{n!}{2^\ell \ell!}.
\end{equation}
This is because the $\ell$ cycles can be chosen in 
\[
\frac{1}{\ell!}\prod_{i=0}^{\ell-1}\binom{n-2i}{2} = \frac{n!}{2^\ell \ell!(n- 2\ell)!}
\] 
many ways, and there are between $a(n-2\ell)!$ and $(n-2\ell)!$ ways to permute the leftover $n-2\ell$ elements without creating any transpositions.
\end{remark}

In what follows assume the family of random variables $(Y_{i,j})_{i,j}$ to satisfy (A1) or (A2) and define 
\begin{equation}\label{eq:number_in_interval}
Z_{n,k}=\bigg|\biggl\{\pi \in S_n: (k-1)n < \sum_{i=1}^nY_{i,\pi(i)} \leq  kn \biggr\}\biggr|.
\end{equation}

\begin{lemma}\label{lemma:estimates_on_z's}
If $(Y_{i,j})_{i,j}$ satisfies (A1) then for all $1 \leq k \leq n$
\begin{equation}\label{eq:Z-estimates}
\left(kn\right)^{n}e^{-(k-1)n} \geq \mathbb{E}(Z_{n,k}) \geq \left\{
\begin{array}{ll}
 e^{-n}n^n, & \text{ if } k=1, \\
e^{-kn}n^n(k-1)^n, & \text{ if } k>1.
\end{array}
\right.
\end{equation}
\end{lemma}

\begin{proof}
For a fixed $\pi$, $\sum_{i=1}^nY_{i,\pi(i)}$ is the sum of $n$ independent exponential random variables with mean $1$, so
\[
\mathbb{E}(Z_{n,k})=\sum_{\pi \in S_n}\mathbb{P}\biggl((k-1)n < \sum_{i=1}^nY_{i,\pi(i)} \leq  kn \biggr) = n!p_n((k-1)n,kn).
\]
Now \eqref{eq:Z-estimates} follows from the estimates in Lemma \ref{lemma:estimates_on_probabilities} (for the lower bound in the case $k>1$ use $k-1 \leq n$).
\end{proof}

In the symmetric case unfortunately, we need to restrict the upper bounds to the permutations satisfying the condition in the event $\mathfrak{B}_n^\epsilon$.
Symmetry of $A$ is not an issue for permutations in $S_n^0$, so first define
\begin{equation}\label{eq:number_in_interval_no_cycles}
Z_{n,k,0}=\bigg|\biggl\{\pi \in S_n^0: (k-1)n < \sum_{i=1}^nY_{i,\pi(i)} \leq  kn \biggr\}\biggr|.
\end{equation}
For $1 \leq k \leq \log n + \log \log n$ and $1 \leq \ell \leq n/2$ let $Z_{n,k,\ell}^\epsilon$ denote the number of permutation $\pi \in S_n^\ell$ satisfying both
\begin{equation}\label{eq:restricting_permutations}
(k-1)n < \sum_{i=1}^n Y_{i,\pi(i)} \leq kn \ \text{ and } \ \sum_{j=1}^rY_{p_j,\pi(p_j)} \leq r(\log n + \log \log n),
\end{equation}
for all $\epsilon n \leq r \leq n$ and all indices $1 \leq p_1 < p_2 < \dots < p_r \leq n$.
Finally set 
\[Z_{n,k}^\epsilon = Z_{n,k,0} + \sum_{1 \leq \ell \leq n/2}Z_{n,k,\ell}^\epsilon.\]

\begin{lemma}\label{lemma:estimates_on_z's-symmetric}
Assume that $(Y_{i,j})_{i,j}$ satisfies (A2). Then there are constants $C>0$ and $a>0$ such that for every $0 < \epsilon \leq  1$ and all $1 \leq k \leq \log n + \log \log n$
\begin{equation}\label{eq:Z-estimates-symmetric-upper}
 \Big(Ce^{-k}kn^{1+\epsilon}\log n\Big)^{n} \geq \mathbb{E}(Z_{n,k}^\epsilon)
\geq \left\{
\begin{array}{ll}
ae^{-n}n^n, & \text{ if } k=1, \\
ae^{-kn}n^n(k-1)^n, & \text{ if } k>1.
\end{array}
\right.
\end{equation}
\end{lemma}

\begin{proof}
The lower bounds follow by 
\[
\mathbb{E}(Z_{n,k}^\epsilon) \geq \mathbb{E}(Z_{n,k,0}) =  \sum_{\pi \in S_n^0}\mathbb{P}\biggl((k-1)n < \sum_{i=1}^nY_{i,\pi(i)} \leq  kn \biggr).
\]
Since the terms $Y_{i,\pi(i)}$, $i=1,\dots, n$ are i.i.d.~when $\pi \in S_n^0$ and using Remark \ref{remark:permutation_with_no_transpositions} to get $|S_{n}^0| \geq an!$, Lemma \ref{lemma:estimates_on_probabilities} bounds the right hand side from below, like in the proof of Lemma \ref{lemma:estimates_on_z's}.
By the same argument
\begin{equation}\label{eq:z-sum_formula}
\mathbb{E}(Z_{n,k,0})
\leq \left(kn\right)^{n}e^{-(k-1)n}.
\end{equation}



If $n$ is even and $\ell=n/2$, then we have $\frac{n!}{2^{n/2}(n/2)!}$ elements in $S_{n}^{n/2}$. For each $\pi \in S_{n}^{n/2}$, the sum $\sum_i Y_{i,\pi(i)}$ is twice the sum of $n/2$ i.i.d.~exponential random variables with rate 1.
Then, ignoring the second condition in \eqref{eq:restricting_permutations}, the upper bound from Lemma \ref{lemma:estimates_on_probabilities} yields similarly as in \eqref{eq:z-sum_formula}
\begin{equation}\label{eq:z-sum_formula_n/2}
\mathbb{E}\Big(Z_{n,k,n/2}^{\epsilon}\Big) \leq \frac{e^{-(k-1)n/2}(kn)^{n/2}n!}{2^{n}(n/2)!^2} \leq e^{-(k-1)n/2}(kn)^{n/2} \leq e^{-(k-1)n}(kn)^{n}.
\end{equation}
In the second inequality above we used the fact that $n!/(n/2)!^2 \leq 2^n$, since the number of $n/2$-element subsets of an $n$-element set is less than the total number of subsets.
For the third inequality we used the fact $e^{-(k-1)}kn \geq 1$, which follows since $ke^{-k} \geq (en)^{-1}$ for all $1 \leq k \leq \log n + \log \log n$. 
This in turn can be checked by observing that the function $k \mapsto ke^{-k}$ is decreasing for $k \geq 1$, and inserting the value $k = \log n + \log \log n$.


Now assume $1 \leq \ell \leq n/2$ and exclude the above case when both $n$ is even and $\ell=n/2$. 
This ensures that $n-2\ell \geq 1$.
For a fixed $\pi \in S_{n}^{\ell}$ we have 
\[
\sum_{i=1}^nY_{i,\pi(i)} = 2R_1 + R_2,
\]
where $R_1$ is the sum $\ell$ independent variables with exponential distribution of rate $1$, and $R_2$ is the sum $n-2\ell$ independent variables with exponential distribution of rate $1$.
Note that for any fixed $\pi$, random variables $R_1$ and $R_2$ are independent.
Next set
\[
j_{n,k,\ell} = \left\{
  \begin{array}{ll}
    \min\{2\ell( \log n + \log \log n)/n,k\}, & \text{ if } \ell > \epsilon n,\\
    \min\{2 \epsilon (\log n + \log \log n),k\}, & \text{ otherwise.}
  \end{array}
\right.
\]
Clearly, if $\pi \in S_n^\ell$ satisfies both conditions in \eqref{eq:restricting_permutations} then $R_1 \leq nj_{n,k,\ell}/2$.
Using the notation from Lemma \ref{lemma:estimates_on_probabilities}, the probability that a permutation $\pi \in S_{n}^{\ell}$ satisfies both conditions in \eqref{eq:restricting_permutations} can be bounded from above by
\begin{align}\label{eq:long_estimates}
& \sum_{1 \leq j \leq j_{n,k,\ell}}p_\ell((j-1)n/2,jn/2)p_{n-2\ell}((k-j-1)n, (k-j+1)n) \nonumber \\
& \leq \sum_{1 \leq j \leq j_{n,k,\ell}}\frac{e^{-(j-1)n/2}(jn/2)^\ell}{\ell!} \cdot \frac{e^{-(k-j-1)n}((k-j+1)n)^{n-2\ell}}{(n-2\ell)!} \nonumber  \\
& = \frac{e^{-(k-3/2)n}n^{n-\ell}}{2^\ell \ell!(n-2\ell)!} \sum_{1 \leq j \leq j_{n,k,\ell}}e^{jn/2}j^\ell (k-j+1)^{n-2\ell}.
\end{align}
Fix some $1 \leq j \leq j_{n,k,\ell}$.
For $\ell \leq \epsilon n$ we have $e^{jn/2} \leq n^{\epsilon n} (\log n)^{\epsilon n}$, while for $\ell > \epsilon n$ we have $e^{jn/2} \leq n^{\ell} (\log n)^{\ell}$.
Note that obtaining these bounds is the reason why we needed to modify the definition on $Z_{n,k}$ from the i.i.d~case.
Using this together with the fact that $j^\ell (k-j+1)^{n-2\ell} \leq k^{n-\ell}$ for any $1 \leq j \leq j_{\ell,k,n}$, the expression in \eqref{eq:long_estimates} can be bounded from above for $n$ large enough by
\[
j_{n,k,\ell}\frac{e^{-(k-3/2)n}(kn)^{n-\ell}}{2^\ell \ell!(n-2\ell)!}\big(n^{\epsilon n}(\log n)^{\epsilon n} + n^\ell(\log n)^\ell\big),
\]
which in turn is no more than
\[ 
\frac{n!}{2^\ell \ell!(n-2\ell)!}(2e^{3/2})^ne^{-kn}k^nn^{(1+\epsilon)n}(\log n)^{n}.
\]
By Remark \ref{remark:permutation_with_no_transpositions} there are $\frac{n!}{2^{\ell}\ell!}$ permutations in $S_{n}^\ell$ so
\[
\mathbb{E}(Z_{n,k,\ell}^{\epsilon}) \leq \frac{n!}{4^\ell (\ell!)^2(n-2\ell)!}(2e^{3/2})^ne^{-kn}k^nn^{(1+\epsilon)n}(\log n)^{n}.
\]
For the first term on the right hand side
\[
\frac{n!}{4^\ell (\ell!)^2(n-2\ell)!} = \binom{n}{2\ell}\binom{2\ell}{\ell}4^{-\ell},
\]
we see that it's less than $2^{n}\cdot 2^{2\ell} \cdot 4^{-\ell} = 2^{n}$.
Therefore,
\[
\mathbb{E}(Z_{n,k,\ell}^{\epsilon}) \leq (4e^{3/2})^ne^{-kn}k^nn^{(1+\epsilon)n}(\log n)^{n}.
\]
Summing the above bound over $1 \leq \ell \leq n/2$, adding \eqref{eq:z-sum_formula} and \eqref{eq:z-sum_formula_n/2} and adjusting the value of the constant $C$, yields the claim.
\end{proof}

\begin{lemma}
\label{prop: BoundZMax}
Under the assumption (A1)  for the array $(Y_{i,j})_{i,j}$ we have
for any $\gamma>1$ and $\lambda > 1$  
\begin{equation}
\label{eq: BoundZMax}
\sum_{n = 2}^{\infty} \mathbb{P}\biggl(\bigcup_{1 \leq k \leq \log n + \frac{ \log \log n}{\lambda}}\left\{Z_{n,k} > \mathbb{E}(Z_{n,k})^{\gamma}\right\}\biggr) < \infty.
\end{equation}
Under the assumption (A2), the convergence \eqref{eq: BoundZMax} holds when $Z_{n,k}$ is replaced by $Z_{n,k}^\epsilon$, for any $\epsilon > 0$.
\end{lemma}

\begin{proof}
First by Markov's inequality 
\begin{align*}
\label{eq: BoundZMarkov}
\mathbb{P}\biggl(\bigcup_{1 \leq k \leq \log n + \frac{\log \log n}{\lambda}}\left\{Z_{n,k} > \mathbb{E}(Z_{n,k})^{\gamma}\right\}\biggr) & \leq \sum_{1 \leq k \leq \log n + \frac{\log \log n}{\lambda}} \mathbb{P}\left(Z_{n,k} > \mathbb{E}(Z_{n,k})^{\gamma}\right)
 \\
& 
\leq \sum_{1 \leq k \leq \log n + \frac{\log \log n}{\lambda}} \frac{1}{ \mathbb{E}(Z_{n,k})^{\gamma-1}}.
\end{align*}
We will use the lower bounds from Lemma \ref{lemma:estimates_on_z's}, to prove that the expression on the right hand side is summable in $n$, which will complete the proof.
Since the lower bounds in Lemma \ref{lemma:estimates_on_z's-symmetric} differ only by a multiplicative constant, the claim for $Z_{n,k}^\epsilon$ follows in the same way.
Therefore, we will restrict our attention to the proof of \eqref{eq: BoundZMax}.

The inequality $\mathbb{E}(Z_{n,1}) \geq e^{-n}n^n$, implies that the series $\sum_{n=1}^{\infty}\mathbb{E}(Z_{n,1})^{1-\gamma}$ converges to a finite limit. 
Therefore we are left to prove 
\begin{equation}
\label{eq: Z2}
\sum_{n =2}^{\infty} \sum_{2 \leq k \leq \log n + \frac{\log \log n}{\lambda}} \frac{1}{ \mathbb{E}(Z_{n,k})^{\gamma-1}} < \infty.
\end{equation}
From the lower bound in  \eqref{eq:Z-estimates} we have that
\begin{equation}\label{eq: BoundUp}
 \sum_{2 \leq k \leq \log n + \frac{ \log \log n}{\lambda}} \frac{1}{ \mathbb{E}(Z_{n,k})^{\gamma-1}}  \leq \frac{e^{n (\gamma-1)}}{n^{n(\gamma-1)}} \sum_{1 \leq k \leq \log n + \frac{ \log \log n}{\lambda}} \frac{e^{kn(\gamma-1)}}{k^{n (\gamma-1)}}.
\end{equation}
The function $g(t)=e^{tn(\gamma-1)}t^{-n (\gamma-1)}$ is convex since $$g''(t)=n(\gamma-1) \frac{e^{tn(\gamma-1)}}{t^{n (\gamma-1) +2}}(n (\gamma-1) (t-1)^2+1) \geq 0.$$
Therefore for any $1 \leq k \leq \log n + \frac{\log \log n}{\lambda}$ we have 
\begin{equation}
\label{eq: gConvex1}
\frac{e^{kn(\gamma-1)}}{k^{n (\gamma-1)}} = g(k) \leq \max\left\{g(1),g\left(\log n + \frac{\log \log n}{\lambda}\right)\right\}.
\end{equation}
For any $n$ large enough we have
$$
g\left(\log n + \frac{\log \log n}{\lambda}\right) = \frac{n^{n(\gamma-1)}(\log n)^{\frac{n (\gamma-1)}{\lambda}}}{(\log n + \frac{\log \log n}{\lambda})^{n (\gamma-1)}} \geq \left(\frac{n}{2(\log n)^{1-1/\lambda}}\right)^{n (\gamma-1)} \geq e^{n (\gamma-1)} = g(1),
$$
and, for such $n$, using (\ref{eq: gConvex1}) we also get
$$
\frac{e^{kn(\gamma-1)}}{k^{n (\gamma-1)}} \leq g\left(\log n + \frac{\log \log n}{\lambda}\right) = \frac{n^{n(\gamma-1)}(\log n)^{\frac{n (\gamma-1)}{\lambda}}}{(\log n + \frac{\log \log n}{\lambda})^{n (\gamma-1)}} \leq \frac{n^{n(\gamma-1)}}{(\log n)^{n(\gamma-1)(1-1/\lambda)}}.
$$ 
Thus, for $n$ large enough (\ref{eq: BoundUp}) yields 
\begin{align*}
\sum_{2 \leq k \leq \log n + \frac{\log \log n}{\lambda}} \frac{1}{ \mathbb{E}(Z_{n,k})^{\gamma-1}} & \leq \frac{e^{n (\gamma-1)}}{n^{n(\gamma-1)}}  \frac{n^{n(\gamma-1)}}{(\log n)^{n(\gamma-1)(1-1/\lambda)}}\left(\log n + \frac{\log \log n}{\lambda}\right) \\
& = \left(\frac{e}{(\log n)^{1-1/\lambda}}\right)^{n(\gamma-1)}\biggl(\log n + \frac{\log \log n}{\lambda}\biggr).
\end{align*}
The expression on the right hand side is summable in $n$ which proves (\ref{eq: Z2}) and thus also (\ref{eq: BoundZMax}).
\end{proof}

Now we are ready to finish the proof of upper bounds.

\begin{proof}[Proof of the upper bounds in Theorem \ref{thm:main} for the Pareto case]
Replacing the elements $\xi_{i,j}$ of $A_n$ by $Y_{i,j} = (\log \xi_{i,j}) /\beta$ the i.i.d~case I) corresponds to (A1) case, and the symmetric case II) corresponds to (A2).

Fix an arbitrary $\gamma > 1$.
In the i.i.d.~case Lemmas \ref{lemma: MaxPerm} and \ref{prop: BoundZMax} and Borel-Cantelli lemma imply that almost surely there exists a positive integer $n_0$ such that for all $n \geq n_0$ we have 
\begin{equation}\label{eq:max_bound_for_large_n}
\max_{\pi \in S_n}\sum_{i=1}^nY_{i,\pi(i)} \leq n \log n + \frac{n \log\log n}{2},
\end{equation}
and that for every $1 \leq k \leq \log n + \frac{\log \log n}{2}$ we have
$Z_{n,k} \leq \mathbb{E}(Z_{n,k})^{\gamma}$.
Therefore, for $n$ large enough we have 
\[
\perm A   =   \sum_{\pi \in S_n} e^{\beta \sum_{i=1}^{n}Y_{i,\pi(i)}} \leq \sum_{1 \leq k \leq \log n + \frac{\log \log n}{2}}e^{\beta kn}Z_{n,k}  \leq  \sum_{1 \leq k \leq \log n + \frac{\log \log n}{2}}e^{\beta kn}\mathbb{E}(Z_{n,k})^{\gamma}.
\]
For the symmetric case, Lemma \ref{lemma: MaxPerm-symmetric} shows that for any $\epsilon$, almost surely there is an integer $n_0$ such that for all $n \geq n_0$ we have $Z_{n,k} = Z_{n,k}^{\epsilon}$, for all $1 \leq k \leq \log n + \frac{1}{2}\log \log n$.
Furthermore, by Lemma \ref{prop: BoundZMax} there is (a possibly larger) $n_0$ for which $n\geq n_0$ implies $Z_{n,k} \leq \mathbb{E}(Z_{n,k}^\epsilon)^{\gamma}$, for all $1 \leq k \leq \frac{1}{2}\log n + \log \log n$.
Then we get as before
\[
\perm A   =   \sum_{\pi \in S_n} e^{\beta \sum_{i=1}^{n}Y_{i,\pi(i)}} \leq \sum_{1 \leq k \leq \log n + \frac{\log \log n}{2}}e^{\beta kn}Z_{n,k}  \leq  \sum_{1 \leq k \leq \log n + \frac{\log \log n}{2}}e^{\beta kn}\mathbb{E}(Z_{n,k}^\epsilon)^{\gamma}.
\]
By the upper bounds in Lemmas \ref{lemma:estimates_on_z's} and \ref{lemma:estimates_on_z's-symmetric}, we have that both $\mathbb{E}(Z_{n,k})$ and $\mathbb{E}(Z_{n,k}^\epsilon)$ are bounded from above by $ \Big(Ce^{-k}kn^{1+\epsilon}\log n\Big)^{n}$, for some $C>1$ which does not depend on $n$ and $\epsilon$.
Therefore, in either case the following is true for $n$ large enough
\begin{align}
\label{eq: PermEstimate}
\perm A  & \leq    \sum_{1 \leq k \leq \log n + \frac{\log \log n}{2}}e^{\beta kn}\Big(Ce^{-k}kn^{1+\epsilon}\log n\Big)^{\gamma n} \nonumber \\
& \leq C^{\gamma n} n^{\gamma (1+\epsilon)n} (\log n)^{\gamma n} \sum_{1 \leq k \leq \log n + \frac{\log \log n}{2}}k^{\gamma n}e^{(\beta - \gamma) kn} \nonumber \\  
& \leq   C^{\gamma n} n^{\gamma (1+\epsilon) n} (\log n)^{\gamma n} \left(\log n + \frac{\log \log n}{2}\right) \max_{1 \leq \tau \leq \log n+\frac{\log \log n}{2 }}\Big(\tau^{\gamma n}e^{ (\beta - \gamma)\tau n}\Big). \nonumber 
\end{align}
If $\beta > 1$ and $\gamma$ is such that $\beta > \gamma > 1$, $\tau^{\gamma n}e^{ (\beta - \gamma)\tau n}$ is an increasing function in $\tau$.  
Thus the upper bound can be obtained by setting $\tau = \log n + \frac{\log \log n}{2}$ above.
Therefore, for $n$ large enough
$$
\perm A \leq   C^{\gamma n}n^{(\beta + \epsilon \gamma) n} (\log n)^{(\beta + \gamma)n/2} \left(\log n + \frac{\log \log n}{2}\right	)^{\gamma n +1}.
$$
This yields
$$
\frac{\log \perm A}{n \log n} \leq \frac{\gamma \log C}{\log n} + \beta + \epsilon \gamma +\frac{(\beta + \gamma) \log \log n}{2\log n} + \left (\gamma + \frac{1}{n}\right)\frac{\log \left(\log n + \frac{\log \log n}{2}\right)}{\log n}.
$$
Since $\epsilon > 0$ was arbitrary
$$
\limsup_{n \to \infty} \frac{\log \perm A}{n \log n} \leq \beta.
$$

In the case $\beta \leq 1$ we want to maximize the function $ \tau^{\gamma n}e^{(\beta - \gamma)\tau n}$. Write $e^{h(\tau)}:=  \tau^{\gamma n}e^{(\beta - \gamma)\tau n}$. We get
 $$h(\tau)=\gamma n \log \tau + (\beta - \gamma) \tau n, \ \ h'(\tau) = \frac{\gamma n}{\tau} + (\beta - \gamma)n, \ \  h''(\tau)=- \frac{\gamma n}{\tau^2} < 0.$$ Therefore, the function $h$ is concave and the maximum occurs when  $h'(\tau)=0$, that is when $\tau = \frac{\gamma}{\gamma-\beta}$ at which the value of the function $e^{h(\tau)}$ is equal to $\left(\frac{\gamma}{(\gamma - \beta)e}\right)^{\gamma n}$.
Now we get
$$
\perm A \leq C^{\gamma n}n^{\gamma (1+\epsilon)n} (\log n)^{\gamma n}\left(\log n + \frac{\log \log n}{2}\right)\left(\frac{\gamma}{\gamma - \beta}\right)^{\gamma n},
$$
and so by taking logarithm as before
$$
\limsup_n \frac{\log \perm A}{n \log n} \leq \gamma(1+\epsilon). 
$$
Since $\gamma >1$ and $\epsilon > 0$ were arbitrary, the claim follows.
\end{proof}

\section{Lower bounds and the proof of Theorem \ref{thm:general_bounds}}\label{sec:lower_bounds}



In this section we prove  Theorem \ref{thm:general_bounds} as well as the lower bounds in Theorem \ref{thm:main}.
An important ingredient is the use of stochastic domination to reduce certain technical issues to matrices with $0,1$ entries.
The following result proven by Hall \cite{Hall48} and Mann and Ryser in \cite{Mann_Ryser53} provides lower bounds for permanents of such matrices (see also Theorem 1.2 in Chapter 4 of \cite{Minc78}).

\begin{proposition}\label{thm:lower_bounds_for_01}
Let $A$ be an $m \times n$ matrix, $m\leq n$ whose all elements are equal to $0$ or $1$. Assume that each row of $A$ contains at least $k$ elements equal to $1$. If $k \geq m$, then
\begin{equation}\label{eq:lower_bounds_on_01_1}
\perm A \geq \frac{k!}{(k-m)!}.
\end{equation}
If $k < m$ and $\perm A > 0$ then
\begin{equation}\label{eq:lower_bounds_on_01_2}
\perm A \geq k!.
\end{equation}
\end{proposition}

As discussed in the introduction (see Remark \ref{rem:matchings}) permanents of matrices with $0$, $1$ elements can be viewed as the number of saturated matchings on corresponding bipartite graphs. 
To ensure the positivity of the permanent, when applying  \eqref{eq:lower_bounds_on_01_2}, we will exploit this connection through the classical Hall's marriage theorem, which can be easily stated in this setting (see \cite{Hall35}).

\begin{theorem}\label{thm:positivity_of_01}
Let $G=(V,E)$ be a bipartite graph and let $V = V_1 \cup V_2$ be a decomposition of  the vertex set so that no two vertices in $V_i$ are connected by an edge, $i=1,2$. Assuming $|V_1| \leq |V_2|$, there exists a saturated matching on $G$ if and only if for any subset $W \subset V_1$ we have $|W| \leq |\{v: v\sim w, w \in W\}|$.
\end{theorem} 

Restating the above theorem in terms of permanents of $0$, $1$ matrices yields the following lemma. 

\begin{lemma}\label{lemma:permanents_and_matchings}
Let $B$ be an $m \times n$ matrix whose all elements are either $0$ or $1$. If for any $1 \leq k \leq m$ any $k \times (n-k+1)$ submatrix of $B$ has at least one element equal to $1$, then $\perm B \geq 1$.
\end{lemma}



All the necessary applications of Proposition \ref{thm:lower_bounds_for_01} and Lemma \ref{lemma:permanents_and_matchings} are summarized in Lemma \ref{lemma:lower_bound_for_submatrices} which, in particular, proves the lower bounds in Theorem \ref{thm:general_bounds}.

\begin{remark}\label{ex:stirling}
Recall that Stirling's formula says that
\[
\lim_{n \to \infty} n!e^nn^{-(n+1/2)} = \sqrt{2\pi}.
\]
In particular there are constants $c_1< c_2$ so that for any $n$ and $1 \leq k \leq n-1$ 
\begin{equation}\label{eq:stirling_for_binom}
c_1 \frac{n^{n+1/2}}{k^{k+1/2}(n-k)^{n-k+1/2}} \leq \binom{n}{k} \leq c_2 \frac{n^{n+1/2}}{k^{k+1/2}(n-k)^{n-k+1/2}}.
\end{equation}
\end{remark}

\begin{lemma}\label{lemma:lower_bound_for_submatrices}
Assume that a sequence of random matrices $A_n$ with positive entries satisfies either I) or II).
For any $0 < \alpha < 1$ and any $\delta > 0$ there exists $r>0$ with the following property: Almost surely there exists $n_0$ such that for any $n \geq n_0$ and any $\alpha n \leq k \leq n$, any $k \times k$ submatrix $B$ of $A_n$ satisfies $\perm B \geq r^k k^{(1-\delta)k}$.
\end{lemma}

\begin{proof}
Let $q>0$ be such that $\mathbb{P}(\xi \leq q) < \eta$, where $\eta$ is to be chosen later.
 Define the random variable $\tilde{\xi} = \mathbf{1}_{(\xi \geq q)}$, and define the matrix $\tilde{A}_n = (\tilde{\xi}_{ij})$. 
Let $\mathfrak{B}_n$ denote the event that  some row of $\tilde{A}_n$ contains more than $\alpha \delta n $ zeros and let $\mathfrak{C}_n$ denote the event that for some $k_1$ and $k_2$ satisfying $\alpha n \leq k_1 + k_2$ there exists a $k_1 \times k_2$ submatrix of $\tilde{A}_n$ containing only zeros. By Lemma \ref{lemma:permanents_and_matchings} on the event $\frak{C}_n^c$ any $k\times k$ submatrix of $\tilde{A}_n$ has a positive permanent, for $\alpha n \leq k \leq n$. Furthermore on the event $\frak{B}_n^c$ every row of every $k \times k$ submatrix of $\tilde{A}_n$ for $k \geq \alpha n$ contains at least $(1-\delta)k$ ones.
Thus,  on the event $\frak{B}_n^c \cap \frak{C}_n^c$ by  \eqref{eq:lower_bounds_on_01_2} we have for any $k \geq \alpha n$ and any $k \times k$ submatrix $B$ of $A_n$
\[
\perm B \geq q^k \perm \tilde{B} \geq q^k \lfloor (1-\delta)k\rfloor! \geq \Big(\frac{q(1-\delta)}{e}\Big)^kk^{(1-\delta)k},
\]
where $\tilde{B}$ is the submatrix of $\tilde{A}_n$ having the same rows and columns as $B$ in $A_n$.
Note that the last inequality above holds for $n$ large enough by Stirling's approximation. 

Thus we only need to prove that the probabilities of the events $\frak{B}_n \cup \frak{C}_n$ are summable (since then they happen only finitely many times almost surely). To end this observe that the average number of $1$s in every row and column of $\tilde{A}_n$ is greater than $n(1-\eta)$, so for $\eta < \alpha \delta$ by standard large deviation arguments there exists a constant $C$ (depending on $\eta$) such that $\mathbb{P}(\frak{B}_n) \leq Cne^{-n/C}$, which is clearly summable. 
For $\frak{C}_n$ use the union bound and observe that for fixed $k_1 \times k_2$ matrix, the probability that it contains only zeros is bounded from above by $\eta^{k_1k_2}$ in the i.i.d.~case and $\eta^{k_1k_2/2}$ in the symmetric case (since we can always extract at least $k_1k_2/2$ independent elements). 
Therefore, for $\eta$ small enough and $n$ large enough
\begin{multline}\label{eq:large_deviations_for _rectangles}
\mathbb{P}(\frak{C}_n) \leq 2\sum_{\alpha n \leq k_1+k_2 \leq n \atop k_1 \geq k_2}\binom{n}{k_1}\binom{n}{k_2}\eta^{k_1k_2/2} \leq  2\sum_{\alpha n/2 \leq k_1 \leq n} \binom{n}{k_1} \sum_{1\leq k_2 \leq n}\binom{n}{k_2}\eta^{k_1k_2/2}  \\ 
\leq 2\sum_{\alpha n/2 \leq k_1 \leq n} \binom{n}{k_1} \Big(\Big(1+\eta^{k_1/2}\Big)^n-1\Big) \leq 2\sum_{\alpha n/2 \leq k_1 \leq n} \binom{n}{k_1} (2\eta^{1/2})^{k_1}.
\end{multline}
To check the last inequality simply observe that for $\eta$ small enough and $n$ large enough
\[
\eta^{-k_1/2}\Big(\big(1+\eta^{k_1/2}\big)^n-1\Big) =  \sum_{\ell=0}^{n-1}\big(1+\eta^{k_1/2}\big)^\ell \leq n\big(1+\eta^{\alpha n/4}\big)^n \leq 2^{\alpha n/2} \leq 2^{k_1}.
\] 
To prove that the right hand side of \eqref{eq:large_deviations_for _rectangles} is summable, observe that for $\eta=\eta(\alpha)$ sufficiently small the following inequalities hold for $\alpha n/2 \leq k_1 \leq n$
\[
(2\eta^{1/2})^{k_1/2} \leq (2\eta^{1/2})^{\alpha n/4} \leq (1-2^{1/2}\eta^{1/4})^{(1-\alpha/2)n} \leq  (1-2^{1/2}\eta^{1/4})^{n-k_1}.
\]
Plugging this back into \eqref{eq:large_deviations_for _rectangles} we get
\[
\mathbb{P}(\frak{C}_n) \leq 2 \sum_{\alpha n/2 \leq k_1 \leq n} \binom{n}{k_1} (2\eta^{1/2})^{k_1/2}(1-2^{1/2}\eta^{1/4})^{n-k_1}.
\]
The right hand side is just twice the probability that the Binomial random variable with parameters $n$ and $2^{1/2}\eta^{1/4}$,  is greater than $\alpha n/2$. Choosing $2^{1/2}\eta^{1/4} < \alpha /2$, large deviation principle implies that this probability is exponentially small, and thus summable in $n$. This finishes the proof.
\end{proof}


\begin{proof}[Proof of Theorem \ref{thm:general_bounds}]
i) Taking an arbitrary $\delta > 0$ by Lemma  \ref{lemma:lower_bound_for_submatrices} we can find $r > 0$ small enough so that almost surely for $n$ large enough
\[
\frac{\log \perm B}{k \log k} \geq \frac{\log r}{\log k} + (1-\delta),
\]
for any $\alpha n \leq k \leq n$ and any $k \times k$ submatrix $B$ of $A_n$. Thus almost surely
\[
\liminf_n \min_{B,k}\frac{\log \perm B}{k \log k} \geq 1-\delta.
\]
Since $\delta > 0$ was arbitrary the claim follows.

ii) 
Let $\tilde{\xi}$ be parameter 1 Pareto distributed random variable.
By Markov inequality, for all $t \geq \mathbb{E}(\xi)$ we have
\[
\mathbb{P}(\xi \geq t) \leq \frac{\mathbb{E}(\xi)}{t} = \mathbb{P}(\mathbb{E}(\xi)\tilde{\xi} \geq t),
\]
Thus  $\mathbb{E}(\xi) \tilde{\xi}$ stochastically dominates $\xi$ from above, and by Remark \ref{rem:stochastic_domination} we can construct  a sequence $(\tilde{A}_n)$ of  random $n \times n$ matrices of the same type as $(A_n)$  (i.i.d.~or symmetric) whose elements are  distributed as $\tilde{\xi}$ and such that the elements of $\mathbb{E}(\xi)\tilde{A}_n$ dominate the elements of $A_n$ pointwise. 
In particular,  for any $\alpha n \leq k \leq n$ and any $k \times k$ submatrix $B$ of $A_n$, for the corresponding submatrix $\tilde{B}$ of $\tilde{A}_n$ we have $\perm B \leq \mathbb{E}(\xi)^k \perm \tilde{B}$. Thus it is enough to prove the claim in the case when we replace $\xi$ with $\tilde{\xi}$. 
Let $\tilde{B}$ be an arbitrary $k \times k$ submatrix of $\tilde{A}_n$ 
and by $\tilde{B}^c$ the matrix at the intersection of the other $n-k$ rows and columns.
Observe that since all the elements are larger than 1 we have $\perm \tilde{B}^c \geq (n-k)!$ and  \[\perm \tilde{A}_n \geq \perm \tilde{B} \perm \tilde{B}^c \geq (n-k)!\perm \tilde{B}.\] 
Therefore,
\[
\frac{\perm \tilde{B}}{k \log k} \leq \frac{\perm \tilde{A}_n}{k \log k} - \frac{\log(n-k)!}{k \log k} 
\leq \frac{\perm \tilde{A}_n}{k \log k} - \frac{(n-k)(\log(n-k) -1) - c}{k \log k},
\]
for some $c> 0$,
where the second inequality inequality follows from Stirling's formula (note that we can assume that $k<n$, since for $k=n$ the upper bounds have been proven in the  previous section). By the upper bounds in Theorem \ref{thm:main}, for any $\epsilon >0$ almost surely there is $n_0$ such that for all $n \geq n_0$ we have 
$\perm \tilde{A}_n \leq (1+\epsilon) n\log n$. For such $n$ we have
\begin{equation}\label{eq:useful_upper_bound}
\frac{\perm \tilde{B}}{k \log k} \leq (1+\epsilon)\frac{n \log n}{k \log k} - \frac{(n-k)\log(n-k)}{k \log k} + \frac{n-k}{k \log k} + \frac{c}{k \log k}
\end{equation}
Since $k \geq \alpha n$ the last two terms on the right hand side vanish in the limit. 
Removing these two terms, the rest of the right hand side of \eqref{eq:useful_upper_bound} can be bounded from above by
\[
\frac{n}{k}\Big((1+\epsilon)\frac{\log n}{\log k} - 1\Big) + \frac{n}{k} -\Big(\frac{n}{k}-1\Big)\frac{\log (n-k)}{\log k}.
\]
The first term above is positive and bounded from above by 
\[
\frac{1}{\alpha} \Big(
\epsilon  + \frac{(1+\epsilon)\log(1/\alpha)}{\log k}
\Big).
\]
Since $\epsilon >0$ is arbitrary,
it suffices to show that
\begin{equation}\label{eq:useful_lowerbound}
\limsup_{n \to \infty}\max_{\alpha n \leq k < n}\Big(\frac{n}{k} -\Big(\frac{n}{k}-1\Big)\frac{\log (n-k)}{\log k}\Big) \leq 1.
\end{equation}
Denoting $n=tk$, where $t > 1$, and using that $s\log s \geq -e^{-1}$ for all $s>0$ we have
\[
\frac{n}{k} -\Big(\frac{n}{k}-1\Big)\frac{\log (n-k)}{\log k} = 1-\frac{(t-1)\log(t-1)}{\log k} \leq 1+\frac{1}{e\log k},
\]
which yields \eqref{eq:useful_lowerbound}.
\end{proof}

\begin{proof}[Proof of the lower bounds in Theorem \ref{thm:main}]
The lower bounds for $\beta \leq 1$ follow from Theorem \ref{thm:general_bounds} i), so in the rest of the proof we will assume that $\beta > 1$.

 First define the random variable $Y =(\log \xi)/\beta$ and observe that 
\[
\lim_{t \to \infty}\frac{\log \mathbb{P}(Y \geq t)}{t} = -1,
\] 
and thus for any $\epsilon >0$ we have $\mathbb{P}(Y \geq t) \geq \exp(-t(1+\epsilon))$, for $t$ large enough. 
If $Y_1, \dots, Y_n$ are independent and distributed as $Y$ and $Q = \max_{1 \leq i \leq n} Y_i$, then for any $t>0$ and $n$ large enough 
\begin{equation}\label{eq:bound_on_Q_general}
\mathbb{P}(Q \leq t\log n) \leq \Big(1 - e^{-(1+\epsilon)t\log n}\Big)^n  =  \Big(1 - n^{-(1+\epsilon)t}\Big)^n \leq \exp(-n^{1-(1+\epsilon)t}).
\end{equation}
Let $(Y_{i,j})$ be an array of random variables distributed as $Y$, whose elements are either i.i.d.~or symmetric, depending on the type of $A$.
It suffices to show that for any  $\epsilon > 0$, almost surely for all $n$ large enough, one can find $(1-2\epsilon)n \leq k \leq (1-\epsilon )n$ and indices $i_1, \dots, i_k$ and $j_1, \dots, j_k$ such that 
\begin{equation}\label{eq:greedy_sufficient}
\sum_{\ell = 1}^k Y_{i_\ell,j_\ell} \geq (1-4\epsilon)n\log n.
\end{equation}
To see this, denote by $B$ the submatrix at the intersection of rows $\{i_1, \dots, i_k\}$ and columns $\{j_1, \dots, j_k\}$ and by $B^c$ the submatrix at the intersection of rows $\{i_1, \dots, i_k\}^c$ and columns $\{j_1, \dots, j_k\}^c$.
For $n$ large enough we then have by \eqref{eq:greedy_sufficient}
\[
\perm B \geq \exp\Big(\beta \sum_{\ell = 1}^k Y_{i_\ell,j_\ell}\Big) \geq n^{(1-4\epsilon)\beta n},
\]
and by Lemma \ref{lemma:lower_bound_for_submatrices}
\[
\perm B^c \geq r^{n-k}(n-k)^{(1-\epsilon)(n-k)} \geq r^{n \epsilon}(n \epsilon)^{n \epsilon},
\]
for some $r>0$ not depending on $n$.
Since $\epsilon > 0$ is arbitrary, the claim follows from 
\[
 \perm A \geq  \perm B  \perm B^c \geq n^{(\beta-4\epsilon\beta + \epsilon) n} (\epsilon r)^{\epsilon n}.
\]
The rest of the proof is devoted to showing the existence of indices $i_1, \dots, i_k$ and $j_1, \dots, j_k$ which yield \eqref{eq:greedy_sufficient}.

To extract these elements, we will run a greedy algorithm.
In the i.i.d.~case, we will have $i_\ell = \ell$, for all $1 \leq \ell \leq k$. 
We start by taking $j_1$ to be a coordinate of the largest element $Q_1$ in the first row, that is $1 \leq j_1 \leq n$ is such that $Q_1 := Y_{1,j_1} \geq Y_{1,\ell}$, for all $1 \leq \ell \leq n$.
Having constructed $j_1, \dots , j_m$, set $j_{m+1}$ to be a coordinate of the largest admissible element $Q_{m+1}$ in the $m+1$-st row, that is $j_{m+1} \in \{1,\dots,n\} \setminus \{j_1,\dots,j_m\}$ is such that 
\[
Q_{m+1} := Y_{m+1,j_{m+1}} \geq Y_{m+1,\ell}, \text{ for all } \ell \in \{1,\dots,n\} \setminus\{j_1, \dots, j_m\}.
\]
Note that, since in each row the location and the value of the maximum are independent, conditioned on the values of $j_1, \dots, j_m$, elements $Y_{m+1,\ell}$, for $\ell \in \{1,\dots,  n\} \setminus \{j_1, \dots , j_m\}$ are independent and distributed as $Y$.
Therefore, $Q_1, \dots, Q_k$ are independent, with $Q_i$ distributed as a maximum of $n-i+1$ i.i.d.~random variables distributed as $Y$.
Here we take $k$ to be any index such that $(1-2\epsilon)n \leq k \leq (1-\epsilon)n$.
In the i.i.d.~case, to prove \eqref{eq:greedy_sufficient} it suffices to show that almost surely 
\begin{equation}\label{greedy_sufficient_iid}
\sum_{i=1}^k Q_i \geq (1-4\epsilon)n \log n,
\end{equation}
holds for $n$ large enough.

To finish the proof in the i.i.d.~case, for a given $\epsilon > 0$ take $n$ large enough so that
\begin{equation}\label{eq:choice_of_n}
(1-4\epsilon)n\log n\leq (1-2\epsilon)^2n\log(n\epsilon).
\end{equation}
Then it is a simple observation that if $\sum_{i=1}^k Q_i < (1-4\epsilon)n \log n$ then we have some $1 \leq r \leq k$ such that 
\[Q_{r} \leq (1-2\epsilon)\log(n\epsilon) \leq (1-2\epsilon)\log(n-r+1).\]
Thus by \eqref{eq:bound_on_Q_general}, for $n$ large enough
\[
\mathbb{P}\Big(\sum_{i=1}^k Q_i < (1-4\epsilon)n \log n\Big) \leq \sum_{1 \leq r\leq (1-\epsilon) n} \exp(-(n-r+1)^{1-(1+\epsilon)(1-2\epsilon)}) \leq n \exp(-(\epsilon n)^{\epsilon + 2\epsilon^2}).
\]
Since the right hand side is summable in $n$, the claim in \eqref{greedy_sufficient_iid} follows.

In the symmetric case we modify the algorithm to extract only the elements strictly above the diagonal. 
Then we reflect the selected elements over the diagonal, to make it appear twice in the sum.
In other words, if $Y_{i,j}$ appears in the sum on the left hand side of \eqref{eq:greedy_sufficient}, so does $Y_{j,i} = Y_{i,j}$.
So set $i_1=1$ and let $j_1$ be a coordinate of the largest element $Q_1'$ in the first row above the diagonal, that is $j_1 \in \{2, \dots, n\}$ is such that $Q_1' := Y_{1,j_1} \geq Y_{1,\ell}$, for all $2 \leq \ell \leq n$.
Having constructed $i_1, \dots, i_m$ and $j_1, \dots , j_m$, set $i_{m+1}$ to be the smallest index such that $i_{m+1} \notin \{i_1, \dots, i_m,j_1, \dots, j_m\}$.
Then set $j_{m+1}$ to be a coordinate  of the largest admissible element $Q_{m+1}'$ in the $i_{m+1}$-st row above the diagonal, that is $j_{m+1} \in \{i_{m+1}+1,\dots, n\}\setminus \{j_1, \dots, j_m\}$ is such that 
\[
Q_{m+1}' := Y_{i_{m+1},j_{m+1}} \geq Y_{i_{m+1},\ell}, \text{ for all } \ell \in \{i_{m+1}+1,\dots, n\}\setminus \{j_1, \dots, j_m\}.
\]
Taking an even $k$ such that $(1-2\epsilon)n \leq k \leq (1-\epsilon)n$, it is an easy observation that among the $k$ elements $\{Y_{i_\ell,j_\ell}, Y_{j_\ell,i_\ell} \ | \ 1 \leq \ell \leq k/2\}$, there are no two in the same row or the same column.
Therefore, in the symmetric case  to prove \eqref{eq:greedy_sufficient} it suffices to show that almost surely 
\begin{equation}\label{greedy_sufficient_symmetric}
\sum_{i=1}^{k/2} Q_i' \geq \frac{1-4\epsilon}{2}n \log n,
\end{equation}
holds for $n$ large enough.
In the $m$-th step of the algorithm, it is clear that $i_m \geq m$.
If $i_{m} = m  + p$, that means that before choosing the $m$-th row, we had to ``skip'' $p$ rows (say $j_{\ell_1}, \dots, j_{\ell_p}$) due to choosing the reflections below the diagonal.
In that case,  $j_{\ell_1}, \dots, j_{\ell_p}$ are exactly all the elements of $j_1, \dots, j_{m-1}$ which are smaller than $i_{m}$. 
Since by the construction there is no $\ell < m$ such that $j_\ell = i_{m}$, we see that the set  $\{i_{m}+1,\dots, n\}\setminus \{j_1, \dots, j_{m-1}\}$ has exactly $n-i_{m} -(m-1-p) = n-2m+1$ elements. 
That means that $Q_{m}'$ is the maximum of $n-2m+1$ independent elements distributed as $Y$, and analogously to the i.i.d.~case, we have that $Q_1', \dots Q_{k/2}'$ are independent.

The rest of the argument goes as before. 
For a fixed $\epsilon > 0$ choose $n$ as in \eqref{eq:choice_of_n}, and observe that $\sum_{i=1}^{k/2} Q_i' < \frac{1-4\epsilon}{2}n \log n$ implies that there is $1 \leq r \leq k/2$ such that $Q_r' \leq (1-2\epsilon)\log(n-2r+1)$. 
As before, using \eqref{eq:bound_on_Q_general} we get for $n$ large enough
\[
\mathbb{P}\Big(\sum_{i=1}^{k/2} Q_i' < (1-4\epsilon)n \log n\Big)  \leq n \exp(-(\epsilon n)^{\epsilon + 2\epsilon^2}),
\]
which is summable in $n$.

\begin{remark}\label{rem:mass_at_zero_proof}
Say we allow $\xi$ to have a point mass at zero. 
Then by the known central limit theorem (see \cite{RempalaWesolowski99}) we know that $\perm A_n /(n\log n)$ converges to 1 in probability (one needs to truncate $\xi$ at some finite value).
In the algorithm from the above proof, one can apply the same argument to the matrix $B^c$, to show that $\perm A_n /(n\log n)$ converges in probability to $\beta$ for $\beta >1$.
To apply justify the application of this argument to $B^c$, one just needs to observe that in our construction, the entries of $B^c$ are independent of the values 
$Q_1, \dots, Q_k$ ($Q_1, \dots, Q_k'$) and are distributed identically to the entries of $A_{n-k}$.
\end{remark}

\end{proof}

\section{Rectangular matrices and the necessity of \eqref{eq:heavy_tail_assumption}}\label{sec:non-square}

In this section we sketch how the above arguments extend to a large class of rectangular matrices. 
Of course, here we will assume that elements are sampled independently from a distribution supported on $\mathbb{R}^+$, but will now allow the width of the matrix to be significantly larger than the height, in particular it will suffice for the height to grow at least as $\log$ of height.
The precise condition under the method extends is that matrix $A_n$ is $m_n \times n$, that is has height $m_n$ and width $n$, and the height satisfies the condition
\begin{equation}\label{eq:condition_on_the_height}
\lim_{n}\frac{m_n \log \log n}{\log n} = \infty.
\end{equation}
Observe that for an $m \times n$ matrix with i.i.d.~elements of mean $\mu$ we have $\mathbb{E}(\perm A_n) = \binom{n}{m} m! \mu^m$ which demonstrates that the  scaling function $n \log n$ will have to be replaced by  $m_n \log n$.

In the whole section we will assume \eqref{eq:condition_on_the_height} and that $(A_n)_n$ is a sequence of $m_n \times n$ matrices on a common probability space with positive elements which are independent and identically distributed as $\xi$.

\begin{theorem}\label{thm:main-rectangular}
Assuming that $\xi$  satisfies \eqref{eq:heavy_tail_assumption}
for some $\beta > 0$,
we have almost surely
\begin{equation}\label{eq:main_result_rectangular}
\lim_{n \to \infty} \frac{\log \perm A_n}{m_n \log n} = \max(1,\beta).
\end{equation}
\end{theorem}

The uniformity over all submatrices of linear size holds as well.

\begin{theorem}\label{thm:general_bounds-rectangular}
Fix $0 < \alpha < 1$.
\begin{itemize}
\item[i)] We have
\begin{equation}\label{eq:general_lower_bounds_rectangular}
\liminf_{n \to \infty} \min_{(k_1, k_2,B)}\frac{\log \perm B}{k_1 \log k_2} \geq 1,
\end{equation}
where the minimum is taken over all  pairs of integers $(k_1, k_2)$ satisfying $\alpha m_n \leq k_1 \leq m_n$, $\alpha n  \leq k_2 \leq n$ and $k_1 \leq k_2$ and all $k_1 \times k_2$ submatrices $B$ of $A_n$.
\item[ii)]
If  $\xi$ has a finite mean then
\begin{equation}\label{eq:general_upper_bounds_rectangular}
\limsup_{n \to \infty} \max_{(k_1, k_2, B)}\frac{\log \perm B}{k_1 \log k_2} = 1,
\end{equation}
where the maximum is taken over all pairs of  integers $(k_1,k_2)$ satisfying $\alpha m_n \leq k_1 \leq m_n$, $\alpha n  \leq k_2 \leq n$ and $k_1 \leq k_2$ and all $k_1 \times k_2$ submatrices $B$ of $A_n$.
\end{itemize}
\end{theorem}

The proofs of these theorems are  modifications of the arguments in the previous two sections, so we will provide sketch of proofs. 
Note that, to simplify notation, we will drop the ceiling and the floor notation throughout the section.

\begin{proof}[Sketch of the proof of the upper bounds in Theorem \ref{thm:main-rectangular}]
As before, by stochastic domination, it suffices to prove the claim when elements are Pareto distributed. To end this one needs to prove a version of Lemma \ref{lemma: MaxPerm} which states that when $(Y_{i,j})$ are independent exponentially distributed with rate one and $\lambda > 1$, we have
\[
\sum_{n = 2}^{\infty}\mathbb{P}\left(\max_{\pi \in S_{m_n,n}}\sum_{i=1}^{m_n} Y_{i,\pi(i)} \geq m_n\log n + m_n \frac{\log\log n}{\lambda}\right) < \infty.
\]
Proceeding as in the proof of Lemma \ref{lemma: MaxPerm} one is left to show that for some fixed $c<1$ and any $\lambda >1$
\[
\sum_{n=2}^\infty \Big((1-c)(\log n)^{c/\lambda}\Big)^{-m_n} < \infty,
\]
which follows from \eqref{eq:condition_on_the_height}.

Next one defines the analog of \eqref{eq:number_in_interval} as
\begin{equation}\label{eq:Z_for_rectangular}
Z_{n,k}=\bigg|\biggl\{\pi \in S_{m_n,n} : (k-1)m_n \leq \sum_{i=1}^{m_n}Y_{i,\pi(i)} <  km_n \biggr\}\biggr|,
\end{equation}
and needs to prove \eqref{eq: BoundZMax}. Calculating expectation of $Z_{n,k}$ 
\begin{equation}\label{eq:expectation_of_Z_rectangular}
\binom{n}{m_n}e^{-km_n}m_n^{m_n}(k-1)^{m_n-1}  \leq \mathbb{E}(Z_{n,k}) \leq \binom{n}{m_n}e^{-(k-1)m_n}(km_n)^{m_n},
\end{equation}
for $k>1$, and $\mathbb{E}(Z_{n,1})\geq \binom{n}{m_n}m_n!e^{-m_n}$ for $k=1$.
The last inequality  handles the sum $\sum_{n \geq 2}\mathbb{E}(Z_{n,1})^{1-\gamma}$. 
We are left to prove the analog of \eqref{eq: BoundUp}, that
\[
 \sum_{2 \leq k \leq \log n + \frac{\log \log n}{\lambda}} \frac{1}{ \mathbb{E}(Z_{n,k})^{\gamma-1}}  \leq \frac{e^{m_n (\gamma-1)}}{\binom{n}{m_n}^{\gamma-1}m_n^{m_n(\gamma-1)}} \sum_{1 \leq k \leq \log n + \frac{\log \log n}{\lambda}} \frac{e^{km_n(\gamma-1)}}{k^{m_n (\gamma-1)}}
\]
is summable in $n$.
Again by the convexity of $g(t)=e^{tm_n(\gamma-1)}t^{-m_n (\gamma-1)}$ and the fact that $g(1) \leq g(\log n + \frac{1}{\lambda} \log \log n)$, proceeding as before one is left to prove that
\[
\left(\frac{ne}{m_n (\log n)^{1-1/\lambda}}\right)^{m_n(\gamma-1)}\frac{\log n + \frac{1}{\lambda}\log \log n}{\binom{n}{m_n}^{\gamma-1}}
\]
is summable in $n$.
By \eqref{eq:condition_on_the_height}, $(\log n)^{-\kappa m_n}$ is summable, for any $\kappa > 0$, and it is enough to show that $\binom{n}{m_n} \geq (cn/m_n)^{m_n}$, for some $c> 0$, which is simple.

To finish the proof assume that both
\begin{align*}
& \max_{\pi \in S_{m_n,n}}\sum_{i=1}^{m_n}Y_{i,\pi(i)} \leq m_n \log n + \frac{m_n \log\log n}{\lambda} \ \ {\rm and}  \\
& Z_{n,k} \leq \mathbb{E}(Z_{n,k})^{\gamma}, \text{ for each } 1 \leq k \leq \log n + \frac{\log \log n}{\lambda},
\end{align*}
hold for some $\lambda > 1$, which is true for $n$ large enough almost surely.
The same  calculations as  in the proof of of the upper bounds in Theorem \ref{thm:main} and \eqref{eq:expectation_of_Z_rectangular} yield
\[
\perm A \leq \binom{n}{m_n}^{\gamma}n^{(\beta-\gamma)m_n}m_n^{m_n\gamma}(\log n)^{\frac{\beta - \gamma}{\lambda}m_n}e^{m_n\gamma}\Big(\log n + \frac{1}{\lambda}\log \log n\Big)^{\gamma m_n}.
\]
Dominant terms are $\binom{n}{m_n}^{\gamma}n^{(\beta-\gamma)m_n}m_n^{m_n\gamma}$. 
Taking logs one sees that it remains to show
\[
\limsup_{n \to \infty}\Big(\frac{\log \binom{n}{m_n}}{m_n \log n} + \frac{\log m_n}{\log n}\Big) \leq 1.
\]
After applying \eqref{eq:stirling_for_binom} we are left with
\[
\limsup_{n} \Big(\frac{n}{m_n}-\Big(\frac{n}{m_n} -1\Big)\frac{\log (n-m_n)}{\log n}\Big) \leq 1.
\]
which  follows from \eqref{eq:useful_lowerbound}. For $\beta \leq 1$ one can repeat the calculations, or simply refer to stochastic domination.
\end{proof}

The proof of Theorem \ref{thm:general_bounds-rectangular} is based on the following equivalent of Lemma \ref{lemma:lower_bound_for_submatrices}.

\begin{lemma}\label{lemma:lower_bound_for_submatrices_rectangular}
For any $0 < \alpha < 1$ and any $\delta > 0$ there exists $r>0$ with the following property: Almost surely there exists $n_0$ such that for any $n \geq n_0$ and any pair of integers $(k_1,k_2)$ satisfying $\alpha m_n \leq k_1 \leq m_n$, $\alpha n \leq k_2 \leq n$, and $k_1 \leq k_2$, any $k_1 \times k_2$ submatrix $B$ of $A_n$ satisfies $\perm B \geq r^{k_1} k_2^{(1-\delta)k_1}$.
\end{lemma}

\begin{proof}[Sketch of the proof of Lemma \ref{lemma:lower_bound_for_submatrices_rectangular}]
One follows the proof of Lemma \ref{lemma:lower_bound_for_submatrices}. In the definitions one needs to write $n$ for the width of the matrix and $m_n$ for the height, for example $\mathfrak{C}_n$  is defined as the event that for some pair of integers $(k_1,k_2)$ satisfying $1 \leq k_1 \leq m_n$, $1 \leq k_2 \leq n$ and $k_1 + k_2 \geq \alpha n$ some $k_1 \times k_2$ submatrix of $A_n$ contains only zeros, and $\mathfrak{B}_n$ is defined as before. On $\mathfrak{B}_n^c \cap \mathfrak{C}_n^c$ one has
\[
\perm B \geq \left\{\begin{array}{ll}
q^{k_1} \lfloor(1-\delta)k_2 \rfloor!, & \text{ for } k_1 \geq (1-\delta)k_2\\
q^{k_1} \frac{\lfloor(1-\delta)k_2 \rfloor!}{(\lfloor(1-\delta)k_2 \rfloor - k_1)!}, & \text{ for } k_1 <(1-\delta)k_2.
\end{array}\right.
\]
In either case, logs of the right hand side is larger than $(1-\delta)k_1\log k_2 + ak_1$, where $a$ is a constant which depends only on $q$ and $\delta$.
Probability of the event $\mathfrak{B}_n$ is estimated as before. For $\mathfrak{C}_n$ one can follow the arguments in \eqref{eq:large_deviations_for _rectangles} starting with
\[
\mathbb{P}(\frak{C}_n)  \leq  2\sum_{\alpha n/2 \leq k_2 \leq n} \binom{n}{k_2} \sum_{1\leq k_1 \leq m_n}\binom{m_n}{k_1}\eta^{k_1k_2}.
\]
This inequality follows from the simple fact that $m_n \leq n$ and $k_1 \leq k_2$ imply that $\binom{m_n}{k_2}\binom{n}{k_1} \leq \binom{m_n}{k_1}\binom{n}{k_2}$.
\end{proof}

\begin{proof}[Sketch of the proof of Theorem \ref{thm:general_bounds-rectangular}]
 Part i) follow from Lemma \ref{lemma:lower_bound_for_submatrices_rectangular}. For part ii) again use Markov's inequality and reduce to the case when elements of $A_n$ are parameter 1 Pareto distributed.
Similarly as before observe that for any  $k_1 \times k_2$ submatrix $B$, any term in the sum defining $\perm B$  can be expanded in $\binom{n-k_1}{m_n - k_1}(m_n-k_1)!$ ways to a term in the sum defining $\perm A_n$. Since all elements of $A_n$ are greater or equal than $1$ we have
\[
\log \perm B  \leq \log \perm A_n - \log \left(\binom{n-k_1}{m_n - k_1}(m_n-k_1)!\right) 
\]
Use proven upper bounds in Theorem \ref{thm:main-rectangular}, Stirling's formula and drop the low order terms to get for $n$ large enough
\[
\log \perm B \leq (1+\epsilon)m_n \log n    + (n-m_n)\log(n-m_n) - (n-k_1)\log(n-k_1)
\]
We are left to prove that
\begin{equation}\label{eq:left_to_prove}
- (n-k_1) \log(n-k_1) + (n-m_n)\log(n-m_n) + (m_n-k_1)\log n \leq o(m_n \log n),
\end{equation}
which is a bit tedious but elementary.
\end{proof}

\begin{proof}[Sketch of the proof of the lower bounds in Theorem \ref{thm:main-rectangular}]
The lower bounds for $\beta \leq 1$ case follow directly from Theorem \ref{thm:general_bounds-rectangular} i). For the case $\beta > 1$ one can follow the arguments almost verbatim. 
Starting from the first row recursively extract the largest admissible elements, and run this greedy algorithm for $k=\rho m_n$ steps.
Extract elements of $(\log \xi_{i,j})/\beta$ whose sum is at least $(1-4\epsilon)m_n\log n$ for $n$ large enough, and apply Lemma \ref{lemma:lower_bound_for_submatrices_rectangular} on the complement submatrix.
\end{proof}

The following example shows that Theorem \ref{thm:main} in general fails when the limit in \eqref{eq:heavy_tail_assumption} does not exist. Actually this is possible at arbitrary small oscillations of the sequence in \eqref{eq:heavy_tail_assumption}. We present the argument for square matrices. 
Note that we will use the fact that the upper bounds (lower bounds) in \eqref{eq:heavy_tail_assumption} imply the upper bound on $\limsup$ (lower bound on $\liminf$) in Theorem \ref{thm:main}.

\begin{example}\label{ex:no_convergence}
Let $S=\{k_i\}$ be a set of positive integers such that $k_{i+1} > 2k_i$. Fix $C_2 >  C_1 > \lambda > 1$ and for every $k \geq 1$ define  the following sequences of positive real numbers 
\[
t_k = \exp(\lambda^k),  \ \tilde{p}_k' = \exp(- \lambda^k/C_1), \text{ and } p_k' = \left\{\begin{array}{ll}\exp(- \lambda^k/C_1), & \text{ for }k \notin S \\
\exp(- \lambda^k/C_2), & \text{ for } k \in S.
\end{array}
\right.
\]
Clearly both series $\sum_k p_k'$ and $\sum_k \tilde{p}_k'$ converge, so we can normalize the sequences with the its sums $Z$ and $\tilde{Z}$ respectively and obtain sequences $p_k = p_k'/Z$ and $\tilde{p}_k = \tilde{p}_k'/\tilde{Z}$.
Let $\xi$ and $\tilde{\xi}$ be random variables supported on the set $\{t_k\}$ with distributions $\mathbb{P}(\xi = t_k) = p_k$ and $\mathbb{P}(\tilde{\xi} = t_k) = \tilde{p}_k$. Observing that the mappings $t \mapsto \mathbb{P}(\xi \geq t)$ and $t \mapsto \mathbb{P}(\tilde{\xi} \geq t)$ are constant on $(t_k , t_{k+1}]$ and that  $\mathbb{P}(\xi \geq t_k) \leq 2p_k$, for all $k \in S$ large enough and for infinitely many $k \notin S$ as well,  and $\mathbb{P}(\tilde{\xi} \geq t_k) \leq 2\tilde{p}_k$, for all $k$ large enough, it is easy to see that
\begin{align}
&\liminf_{t \to \infty} \frac{\log \mathbb{P}(\xi \geq t)}{\log t} = -\frac{1}{C_1/\lambda} < -\frac{1}{C_2} = \limsup_{t \to \infty} \frac{\log \mathbb{P}(\xi \geq t)}{\log t}, \label{eq:example_bounds_1} \\
&\liminf_{t \to \infty} \frac{\log \mathbb{P}(\tilde{\xi} \geq t)}{\log t} = -\frac{1}{C_1/\lambda} < -\frac{1}{C_1} = \limsup_{t \to \infty} \frac{\log \mathbb{P}(\tilde{\xi} \geq t)}{\log t} \label{eq:example_bounds_2}.
\end{align}

As usual let $(A_n)$ denote a sequence of $n \times n$ matrices on a common probability space with independent elements distributes as $\xi$.  We will prove that
\begin{equation}\label{eq:example_statement}
\liminf_n \frac{\log \perm A_n}{n \log n} \leq C_1, \text{ and } \limsup_n \frac{\log \perm A_n}{n \log n} = C_2
\end{equation}

To get the upper bound on $\liminf$ take a sequence $(\ell_i)$ of positive integers such that $k_i < \ell_i < k_{i+1}$ and that sequences 
 $(\ell_i - k_i)$ and $(k_{i+1} -\ell_i)$ are strictly increasing.
Define integers $n_i = \exp(\lambda^{\ell_i}/C_1)$. By a simple union bound the probability that $A_{n_i}$ contains an element $t_\ell$, for some $\ell \geq k_{i+1}$ is bounded from above by 
\[
2n_i^2p_{k_{i+1}} =\frac{1}{Z}\exp(2\lambda^{\ell_i}/C_1) 2\exp(-\lambda^{k_{i+1}} /C_2),
\]
for $i$ large enough. This expression is summable in $i$, so almost surely $A_{n_i}$ does not contain elements greater than $t_{k_{i+1}}$ for $i$ large enough. Thus to prove the first inequality in \eqref{eq:example_statement} one can assume that elements in $A_{n_i}$ are distributed as $\xi\mathbf{1}_{(\xi < t_{k_{i+1}})}$. 
Next observe that $\xi\mathbf{1}_{(\xi < t_{k_{i+1}})}$ is stochastically dominated by $t_{k_i}\tilde{\xi}$, that is 
\[
\mathbb{P}(t \leq \xi < t_{k_{i+1}}) \leq \mathbb{P}(t_{k_i}\tilde{\xi} \geq t).
\]
While the inequality is trivial for $t\geq t_{k_{i+1}}$ and for $t \leq t_{k_i}$, for $t_{k_i} < t < t_{k_{i+1}}$ it follows from the fact that  for $J \subset (t_{k_i},t_{k_{i+1}})$ 
\[
\mathbb{P}(\xi \in J) =\frac{\sum_{k : t_k\in J}p_k'}{Z}  \leq \frac{\sum_{k : t_k\in J}\tilde{p}_k'}{\tilde{Z}} = \mathbb{P}(\tilde{\xi} \in J),
\] 
since $p_k' = \tilde{p}_k'$, for $t_k \in J$ and $\tilde{Z} \leq Z$.
Thus if $\tilde{A}_n$ is the sequence of $n \times n$ matrices whose elements are identical and distributed as $\tilde{\xi}$ then
\[
\liminf_{n} \frac{\log \perm A_n}{n \log n} \leq \limsup_i \frac{n_i\log t_{k_i} + \log \perm \tilde{A}_{n_i}}{n_i \log n_i} \leq \lim_i \frac{\lambda^{k_i}}{\lambda^{\ell_i}/C_1} + C_1 = C_1.
\]
Here the second inequality follows form the upper bounds in Theorem \ref{thm:main} and \eqref{eq:example_bounds_2}.

To prove the second relation in \eqref{eq:example_statement} fix $k\in S$, $\epsilon > 0$ and define the integer  $n = n_k = \exp((1+\epsilon)\lambda^k/C_2)$. We proceed with a greedy algorithm analogous to the one in the proof of the lower  bound in Theorem \ref{thm:main}. With the probability $1- (1-p_k)^n$  there is an element in the first row of $A^{(0)} = A_n$ equal to $t_k$.  On this event take the first such element, remove the corresponding column and the first row from $A_n$ and obtain the $(n-1) \times (n-1)$ matrix $A^{(1)}$ which is independent of the first row and is distributed as $A_{n-1}$. Now repeat the step with $A^{(1)}$ instead of $A^{(0)}$ and proceed recursively as long as one is successful at each step. For $0 < \rho < 1$ one will not be able to proceed till step $\rho n $ with probability at most
\[
\sum_{(1-\rho) n \leq i \leq  n}(1-p_k)^i \leq \frac{(1-p_k)^{(1-\rho)n}}{p_k} \leq 3Z\exp\Big(-\frac{1-\rho}{Z}\exp(\epsilon\lambda^k/C_2)+\lambda^k/C_2\Big),
\]
if $k$ is chosen large enough. The right hand side is clearly summable in $k$ and thus almost surely for $k$ large enough and $n=n_k$ constructed as above, the above algorithm will be successful for $\rho n$ steps. In that case one can get lower bound on $\perm A_n$ as in the proof of the lower bounds in Theorem \ref{thm:main}: The matrix at the intersection of the first $\rho n$ rows and the removed columns is bounded from below by the product of extracted elements, that is $t_k^{\rho n}$ and the matrix at the intersection of the last $(1-\rho)n$ rows and non-removed columns is bounded from below by $((1-\rho) n)!$ (since all of it's elements are greater or equal than $1$). Therefore for $k$ large enough and $n$ constructed as above $\perm A_n \geq t_k^{\rho n} ((1-\rho) n)!$. Since $\log ((1-\rho) n)!/(n \log n) \to 1-\rho$ and
\[
\frac{\log t_k^{\rho n}}{n \log n}  \geq \frac{\rho\log t_k}{\log n} \geq \frac{\rho \lambda^k }{(1+\epsilon)\lambda^k  /C_2} \to \frac{C_2 \rho}{1+\epsilon},
\]
we obtain that almost surely
\[
\limsup_{n \to \infty} \frac{\log \perm A_n}{n \log n} \geq \frac{C_2 \rho}{1+\epsilon} + 1- \rho.
\]
By sending $\rho \to 1$ and $\epsilon \to 0$ we get $C_2$ as the lower bound on the $\limsup$, and by upper bounds in Theorem \ref{thm:main} and \eqref{eq:example_bounds_1} it is equal to $C_2$.
\end{example}

\textbf{Acknowledgments}
The author would like to thank Professor Sourav Chatterjee for suggesting this problem and for generous help and guidance with this paper.


\end{document}